\documentclass[12pt]{amsart} 

\usepackage{amsmath,amssymb,amsthm,amsfonts,enumitem}
\usepackage{color}
\usepackage{graphicx,subcaption }
\usepackage{hyperref}
\usepackage{epstopdf}
\usepackage{bm}
\usepackage{todonotes}

\theoremstyle{plain} 
\newtheorem{thm}{Theorem}

\newtheorem{prop}[thm]{Proposition}

\theoremstyle{definition}

\newtheorem{conj}{Conjecture}

\theoremstyle{remark}
\newtheorem{rem}[thm]{Remark}

\def\RR{\mathbb{R}}

\begin{document}

\title{A note on the affine-invariant plank problem}

\author{Gregory R. Chambers}
\address{Department of Mathematics, Rice University, Houston, TX}
\email{gchambers@rice.edu}

\author{Lawrence Mouill\'e}
\address{Department of Mathematics, Trinity University, San Antonio, TX}
\email{lmouille@trinity.edu}

\begin{abstract}
Suppose that $C$ is a bounded, convex subset of $\RR^n$, and that $P_1, \dots, P_k$ are planks
which cover $C$ in respective directions $v_1, \dots, v_k$ and with widths
$w_1, \dots, w_k$.  In 1951, Bang conjectured that the sum of relative widths
$$\sum_{i=1}^k \frac{w_i}{w_{v_i}(C)} \geq 1, $$
generalizing a previous conjecture of Tarski.
Here, $w_{v_i}(C)$ is the width of $C$ in the direction $v_i$.

In this note we give a short proof of this conjecture under the assumption that, for every $m$ with $1 \leq m \leq k$,
$ C \setminus \bigcup_{i = 1}^m P_i $
is a convex set.  
In addition, we prove that if the projection of $C$ onto the vector space spanned by the normal vectors of the planks has dimension $d$, then the above sum of relative widths is at least $1/d$.
\end{abstract}
\maketitle


A \emph{plank} $P$ is the closed, connected region between two parallel hyperplanes in $\RR^n$.
The distance between these supporting hyperplanes is called the \emph{width} of $P$.  
We define the \emph{direction} of $P$ to be the one-dimensional subspace $v$ of $\RR^n$ which is perpendicular to the supporting hyperplanes of $P$.
In 1932, Tarski proved in \cite{tarski} that if $P_1, \dots, P_k$ are planks with widths $w_1, \dots, w_k$ which cover a unit ball in $\RR^2$,
then
$$ \sum_{i=1}^k w_i \geq 2. $$

He asked if an analogous result was true for every bounded, convex subset $C$ of $\RR^n$.  For any one-dimensional subspace $v$ of $\RR^n$, we define the width of $C$ in the direction $v$,
$w_v(C)$, as the width of the smallest plank of direction $v$ which covers $C$.  In other words, if we choose hyperplanes perpendicular to $v$, which do not intersect the interior of $C$,
and which are minimally separated, then $w_v(C)$ is the distance between those hyperplanes.  Tarski's question can be thus stated as follows:  If $P_1, \dots, P_k$ are planks which
cover $C$ with widths $w_1, \dots, w_k$, then does
$$ \sum_{i=1}^k w_i \geq \inf_{v \in G(n,1)} w_v(C) $$
hold? Here, $G(n,1)$ is the Grassmannian of one-dimensional linear subspaces in $\RR^n$.

This conjecture was resolved in 1951 by Bang in \cite{bang}.  
In the same article, Bang posed the following well-known affine-invariant generalization of Tarski's problem:

\begin{conj}[Bang \cite{bang}]
	\label{conj:affine_invariant}
	Suppose that $C$ is a convex, bounded subset of $\RR^n$, and that $P_1, \dots, P_k$ are planks in directions $v_1, \dots, v_k$ and of widths $w_1, \dots, w_k$
	which cover $C$.  Then
	$$ \sum_{i=1}^k \frac{w_i}{w_{v_i}(C)} \geq 1. $$
\end{conj}
If $w_{v_i}(C) = 0$ for some $i$, then we define the corresponding ratio as $+ \infty$, and so the inequality is trivially true.

This conjecture is known to be true for $n=2$ \cite{BangDim2} or when the set $C$ is centrally symmetric \cite{ball}, and a strengthened version of the conjecture has recently been proven for dimensions $n \leq 14$ \cite{Pinasco}.
For a survey of related results, as well as other variations of this problem,
see \cite{Bezdek2013}.
In this note, we prove Conjecture \ref*{conj:affine_invariant} under an additional assumption:
\begin{thm}
	\label{thm:main}
	Suppose that $C$ is a convex, bounded subset of $\RR^n$, and that $P_1, \dots, P_k$ are planks which cover $C$.  Furthermore, assume that
	$$ C \setminus \bigcup_{i=1}^m P_i $$
	is convex for every $m$ with $1 \leq m \leq k$.  Then Conjecture \ref*{conj:affine_invariant} is true for $C$ and $P_1, \dots, P_k$.
\end{thm}

Now, recall that the dimension of a convex set $C\subseteq \mathbb{R}^n$ is defined to be the dimension of the unique affine subspace of $\mathbb{R}^n$ containing $C$ in which $C$ has non-empty interior.
We also prove the following theorem, which gives a lower bound on the sum of relative widths that is not sharp, but does hold for general configurations of planks. 


\begin{thm}
	\label{thm:main2}
	Suppose that $C$ is a convex, bounded subset of $\RR^n$, and that $P_1, \dots, P_k$ are planks in directions $v_1, \dots, v_k$ and of widths $w_1, \dots, w_k$
	which cover $C$. 
	Define $V$ to be the span of the directions $v_1,\dots,v_k$, and define $\pi_V(C)$ to be the orthogonal projection of $C$ onto $V$. 
	Then
	\[
		\sum_{i=1}^k \frac{w_i}{w_{v_i}(C)} \geq \frac{1}{\dim \pi_V(C)}.
	\]
\end{thm}

Theorem~\ref*{thm:main2} is considered a folklore result within the community, though the authors could not find it in the literature.
See Remark~\ref{rem:improvement} below for a further refinement of Theorem~\ref*{thm:main2}.

\section{Proof of Theorem~\ref{thm:main}}

The proof of Theorem~\ref{thm:main} follows immediately from repeatedly applying the following proposition.  We note that the
idea of dilating a convex set about a point on the boundary used here was also used by Bang in \cite{bang} and Alexander in \cite{alexander} to prove different results.
\begin{prop}
	\label{prop:removal}
	Suppose that $C$ is a bounded, convex subset of $\RR^n$, and that $P_1, \dots, P_k$ are planks which cover $C$.  
	Furthermore, suppose that $P_1$ has the
	property that $C \setminus P_1$ is a convex set.  
	If Conjecture \ref{conj:affine_invariant} is true for $C \setminus P_1$ covered by $P_2, \dots,  P_k$, then it is true for $C$ covered by $P_1, P_2, \dots, P_k$.
\end{prop}
\begin{proof}
	Define $X = C \setminus P_1$.  We first show that we may make the following assumptions:
	\begin{enumerate}
		\item	$w_{v_i}(C) > 0$ for all $1 \leq i \leq k$.
		\item	$X$ is not empty.
		\item	$w_{v_i}(X) > 0$ for all $1 \leq i \leq k$.
	\end{enumerate}
	If $w_{v_i}(C) = 0$ for some $i$, then $\frac{w_i}{w_{v_i}(C)} = +\infty > 1$, completing the proof.
	If $X$ is empty, then $\frac{w_1}{w_{v_1}(C)} \geq 1$, finishing the proof.
	If $w_{v_i}(X) = 0$ for some $i$, then since $X$ contains a point which is not in $P_1$, since $P_1$ is closed, and since $C$ is convex,
	$w_{v_i}(C) = 0$, contradicting the first assumption.

	Suppose $X$ and $C$ satisfy all of the above assumptions.  We will prove that, for every $v \in G(n,1)$,
	\begin{equation}
		\label{eqn:main_inequality}
		\frac{w_v(X)}{w_v(C)} \geq \frac{w_{v_1}(X)}{w_{v_1}(C)} \geq 1 - \frac{w_1}{w_{v_1}(C)}. \tag{$\ast$}
	\end{equation}
	The second inequality is immediate, since $w_{v_1}(X) \geq w_{v_1(C)} - w_1$.  We prove the first inequality, beginning with the statement of some elementary facts about the
	widths of convex sets.  Here, $v$ is any element of $G(n,1)$, and $K_1$, $K_2$, and $K$ are any bounded convex sets in $\RR^n$.
	\begin{enumerate}
		\item	If $K_1$ and $K_2$ are translates of each other, then $w_v(K_1) = w_v(K_2)$.
		\item	If $K_1 = c K_2$ for some $c \in \RR$, then $w_v(K_1) = |c| w_v(K_2)$.
		\item	If $K_1 \subset K_2$, then $w_v(K_1) \leq w_v(K_2)$.
		\item	$w_v(K) = w_v(\overline{K})$
	\end{enumerate}

	Our proof proceeds by considering two cases based on how $w_{v_1}(X)$ compares to $w_{v_1}(C)$.

	\noindent {\bf Case 1:} If $w_{v_1}(X) = w_{v_1}(C)$, then $\overline{X} = \overline{C}$, and so from fact $(4)$ we have that $w_v(X) = w_v(C)$ for all $v \in G(n,1)$.  As a result,
	$\frac{w_v(X)}{w_v(C)} \geq 1 = \frac{w_{v_1}(X)}{w_{v_1}(C)}.$
	
	\noindent {\bf Case 2:} If $w_{v_1}(C) > w_{v_1}(X)$, then one of the supporting hyperplanes of $X$ in the direction $v_1$
	lies in $P_1$, and the other lies outside of it.  Let $H_1$ be the first hyperplane, and let $H_2$ be the second; we have that $H_1 \cap P_1 = \varnothing$, and that $H_2 \subset P_1$. Since $H_1$ and $H_2$ are supporting
	hyperplanes, there exists a point $p \in H_1 \cap \partial C$.  From fact $(1)$, we may assume that $p$ is the origin.
	Let $P^*$ be the plank between $H_1$ and $H_2$.  From the definition of $X$, and the above assumptions about $H_1$ and $H_2$, $$P^* \cap \overline{C} = \overline{X}.$$  In particular,
	if $V$ is the unit vector perpendicular to $H_1$ which points into $P^*$, then $x \in \overline{X}$ if and only if $0 \leq V \cdot x \leq w_{v_1}(X)$.

	Let $\rho = \frac{w_{v_1}(X)}{w_{v_1}(C)}$, and consider $\rho C$.  We claim that $\rho C \subset \overline{X}$.  From the above inequality, we need only show that
	$$ 0 \leq V \cdot x \leq w_{v_1}(X)$$
	for every $x \in \rho C$.  For every such $x$, there is a $y \in C$ such that $x = \rho y$.  We then have that
	$$ 0 \leq y \cdot V \leq w_{v_1}(C),$$
	and so
	$$ 0 = 0 \rho \leq x \cdot V \leq \rho w_{v_1}(C) = w_{v_1}(X).$$
	Since $\rho C \subset \overline{X}$, by facts (2), (3), and (4), $\rho w_v(C) \leq w_v(X)$ for every $v \in G(n,1)$.  We obtain our result by rearranging this inequality.
	This proof is illustrated in Figure \ref*{fig:dilate}.
	
	\begin{figure}[h]
		\centering
		\includegraphics[width=1.00\textwidth]{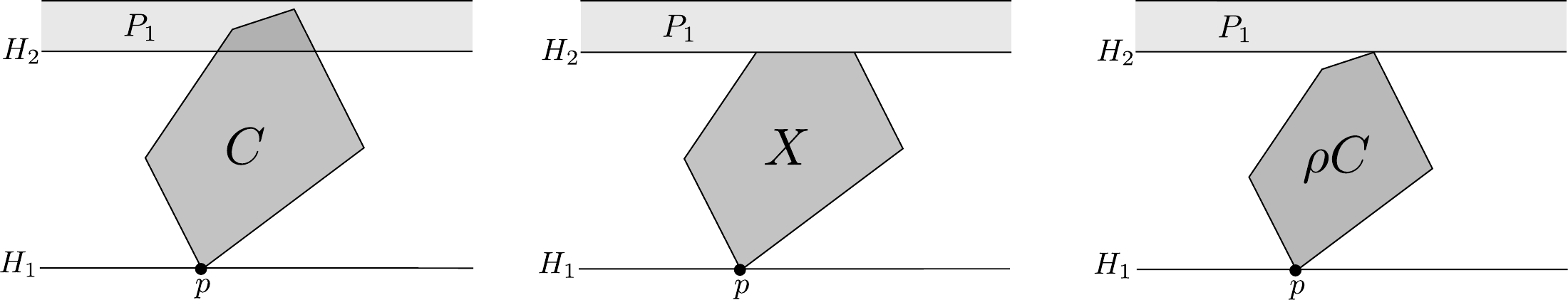}
		\caption{The homothetic image $\rho C$ of the convex set $C$ is contained in the closure of $X = C \setminus P_1$.}
		\label{fig:dilate}
	\end{figure}
	
	Now, suppose that Conjecture \ref{conj:affine_invariant} holds for $X$.  We then have that
	$$ \sum_{i=2}^k \frac{w_i}{w_{v_i}(X)} \geq 1, $$
	and so by multiplying both sides by $\frac{w_{v_1}(X)}{w_{v_1}(C)}$ and by using ($\ast$),
	$$ \sum_{i=2}^k \frac{w_i}{w_{v_i}(C)} \geq \frac{w_{v_1}(X)}{w_{v_1}(C)}. $$
	Adding the contribution from $P_1$ to both sides and using ($\ast$) again, we have that
	\[
		 \sum_{i=1}^k \frac{w_i}{w_{v_i}(C)} \geq \frac{w_1}{w_{v_1}(C)} + \frac{w_{v_1}(X)}{w_{v_1}(C)}  \geq 1. \qedhere
	\]
\end{proof}

\begin{rem}\label{rem:Hunter}
	We mention a noteworthy cover of a convex set by planks which cannot be ordered as in Theorem \ref{thm:main}.
	As Hunter observed in \cite{hunter},
	there is a configuration of three planks $P_1, P_2$, and $P_3$ which cover the equilateral triangle $T$ of side length $1$ such that no plank's removal results in a convex set, and such that
	$$ \frac{w_1}{w_{v_1}(T)} + \frac{w_2}{w_{v_2}(T)} + \frac{w_3}{w_{v_3}(T)} = 1.$$
	The configuration is as follows:  
	Suppose that the barycenter of the equilateral triangle is the origin, and that one edge is parallel to the $x$-axis.  
	Each of our three planks has width $1/3$, and they are centered on the lines passing through the origin with tangent vectors $(0,1)$, $(\sqrt{3},-1)$, and $(-\sqrt{3},-1)$.
	In particular, we cannot hope to prove if a configuration of planks covering a convex body has the sum of their relative widths equal to $1$, then they must satisfy the hypotheses of Theorem \ref{thm:main}.
\end{rem}

\section{Proof of Theorem~\ref{thm:main2}}

\begin{proof}
	Suppose that $C$ is a convex, bounded subset of $\RR^n$, and that $P_1, \dots, P_k$ are planks in directions $v_1, \dots, v_k$ and of widths $w_1, \dots, w_k$
	which cover $C$. 
	
	First, assume that the vectors $v_1, \dots, v_k$ span $\RR^n$ and that $C$ has non-empty interior as a subset of $\RR^n$, meaning it has dimension $n$.
	We wish to show that the sum of the relative widths of $C$ is at least $\frac{1}{n}$.
	
	Recall that the John ellipsoid for a given $d$-dimensional convex set $K$ is the unique ellipsoid contained in $K$ with maximal $d$-dimensional volume (see \cite{john}).  
	In addition, if the John ellipsoid for $K$ is $B^d_1$, the unit ball centered at the origin in $\RR^d$, then $K$ is contained in $B^d_d$, the ball of radius $d$ centered at the origin in $\RR^d$.
	
	Now, since the sum of relative widths remains invariant under affine transformations, we can apply such a transformation to $\mathbb{R}^n$ so that the image of $C$ has the unit ball $B^n_1$ as its inscribed John ellipsoid.
	Thus, the image of $C$ under this affine transformation is contained in the ball $B^n_n$ of radius $n$.
	Because applying this affine transformation does not affect the sum of the relative widths, we will use ``$C,P_1,\dots,P_k$'' to refer to their respective images under this transformation.
	
	Observe that, since the planks $P_1,\dots,P_k$ cover $C$, they also cover $B^n_1$.
	By Bang's solution to the original Tarski plank problem \cite{bang}, it follows that the total width of all the planks, $w_1+\cdots+ w_k$, is at least $2$ (twice the radius of the ball).
	Because the ball $B^n_n$ contains $C$, each width $w_{v_i}(C)$ is at most $2n$.
	Therefore, 
	\[
	\sum_{i=1}^k \frac{w_i}{w_{v_i}(C)} \geq \frac{1}{2n} \sum_{i=1}^k w_i \geq \frac{1}{n}.
	\]
	
	
	Now assume that the directions $v_1, \dots, v_k$ do not span $\RR^n$ or that $C$ has empty interior in $\RR^n$.
	Define $V$ to be the span of $v_1,\dots,v_k$, and define $\pi_V:\RR^n\to V$ to be the orthogonal projection onto $V$. 
	Notice because $C$ is convex in $\RR^n$, we have that $\pi_V(C)$ is convex in $V$.
	Furthermore, the projections $\pi_V(P_1), \dots , \pi_V(P_k)$ are planks in $V$ that cover $\pi_V(C)$.
	Given $i\in\{1,\dots,k\}$, the plank $\pi_V(P_i)$ has width and direction equal to that of $P_i$, namely $w_i$ and $v_i$, respectively.
	Also, the relative widths of $\pi_V(C)$ in $V$ are equal to those of $C$ in $\RR^n$, meaning $w_{v_i}(\pi_V(C)) = w_{v_i}(C)$ for each $i$.
	
	Define $d$ to be the dimension of $\pi_V(C)$.
	In particular, because $\pi_V(C)$ is convex in $V$, there exists a unique $d$-dimensional affine subspace $A\subseteq V\subseteq \RR^n$ containing $\pi_V(C)$ in which $\pi_V(C)$ has non-empty interior.
	Let $C',P_1',\dots,P_k'$ denote the intersections of $\pi_V(C), \pi_V(P_1),\dots,\pi_V(P_k)$ with $A$.
	Then $w_{v_i}(C') = w_{v_i}(\pi_V(C)) = w_{v_i}(C)$.
	As in the proof of Proposition \ref{prop:removal}, we may assume the relative width $w_{v_i}(C') > 0$ for all $i$.
	It follows that each direction $v_i$ is not orthogonal to $A$, and hence $P_1',\dots,P_k'$ are planks in $A$.
	For each $i$, define $v_i'$ and $w_i'$ to be the direction and width of $P_i'$ in $A$, respectively.
	By a similar triangles argument, it is elementary to show that the relative widths of $C'$ with respect to $\pi_V(P_i)$ and with respect to $P_i'$ are equal.
	In other words, for each $i$, 
	\[
		\frac{w_i}{w_{v_i}(C')} = \frac{w_i'}{w_{v_i'}(C')}.
	\]
	
	Finally, because $v_1',\dots,v_k'$ span $A$ and $C'$ has non-empty interior in $A$, it follows from the argument above that 
	\[
		\sum_{i=1}^k \frac{w_i}{w_{v_i}(C)} = \sum_{i=1}^k \frac{w_i'}{w_{v_i'}(C')} \geq \frac{1}{d} = \frac{1}{\dim \pi_V(C)}. \qedhere
	\]
\end{proof}

\begin{rem}
	\label{rem:improvement}
	Instead of using the John ellipsoid, one can consider the set $\hat{C} = (C-g) \cap (g-C)$, where $g$ denotes the centroid (center of mass) of $C$.
	$\hat{C}$ is convex, symmetric about the origin, and can be inscribed in $C$ via a translation.
	In particular, Ball's result from \cite{ball} can be applied to the sum of relative widths for $\hat{C}$.
	On the other hand, by a classical theorem of Minkowski and Radon, for every direction $v_i$, the width $w_{v_i}(\hat{C})$ is at least $\frac{2}{1+d} w_{v_i}(C)$, where $d$ denotes the dimension of $C$.
	From these observations, it follows that the sum of relative widths of $C$ is at least $\frac{2}{1+d}$, which improves upon the bound in Theorem \ref{thm:main2}.
	We thank an anonymous referee for sharing this observation with us.
\end{rem}

\begin{rem}
	It appears that the proof of Theorem \ref{thm:main2} can also be applied in the context of the \textit{fractional plank problem}; see \cite{Aharoni_2002} for a description of this problem.
\end{rem}

\bigskip
\noindent {\bf Acknowledgments} \quad The authors would like to thank Marianna C\"{o}rnyei and Almut Burchard for useful conversations with regard to this problem and anonymous referees for their feedback.  The first author was supported in part by NSF grant DMS-1906543, and the second author was supported in part by NSF award DMS-2202826.

\bibliographystyle{amsplain}
\bibliography{bibliography}

\end{document}